\documentclass[12pt,reqno,a4paper]{amsart}

\usepackage{hyperref}

\usepackage{comment}

\usepackage{graphicx}

\usepackage{amsmath,amsfonts,amsthm,mathrsfs,amssymb,cite}
\usepackage[usenames]{color}

\newtheorem{thm}{Theorem}[section]

\newtheorem{lem}{Lemma}[section]
\newtheorem{prop}{Proposition}[section]

\theoremstyle{definition}
\newtheorem{defn}{Definition}[section]
\theoremstyle{remark}

\newtheorem{rem}{Remark}[section]
\numberwithin{equation}{section}

\title[Nonlocal inverse problem for Schr\"odinger operator]{Simultaneously recovering potentials and embedded obstacles for anisotropic fractional Schr\"odinger operators}

\author{Xinlin Cao}
\address{Department of Mathematics, Hong Kong Baptist University, Kowloon, Hong Kong SAR, China}
\email{xlcao.math@foxmail.com}

\author{Yi-Hsuan Lin}
\address{Department of Mathematics, University of
Washington, Seattle, USA}
\email{yihsuanlin3@gmail.com}

\author{Hongyu Liu}
\address{Department of Mathematics, Hong Kong Baptist University, Kowloon, Hong Kong SAR, China}
\email{hongyu.liuip@gmail.com; hongyuliu@hkbu.edu.hk}

\begin{document}

\begin{abstract}
Let $A\in\mathrm{Sym}(n\times n)$ be an elliptic 2-tensor. Consider the anisotropic fractional Schr\"odinger operator $\mathscr{L}_A^s+q$, where $\mathscr{L}_A^s:=(-\nabla\cdot(A(x)\nabla))^s$, $s\in (0, 1)$ and $q\in L^\infty$. We are concerned with the simultaneous recovery of $q$ and possibly embedded soft or hard obstacles inside $q$ by the exterior Dirichlet-to-Neumann (DtN) map outside a bounded domain $\Omega$ associated with $\mathscr{L}_A^s+q$. It is shown that a single measurement can uniquely determine the embedded obstacle, independent of the surrounding potential $q$. If multiple measurements are allowed, then the surrounding potential $q$ can also be uniquely recovered. These are surprising findings since in the local case, namely $s=1$, both the obstacle recovery by a single measurement and the simultaneous recovery of the surrounding potential by multiple measurements are longstanding problems and still remain open in the literature. Our argument for the nonlocal inverse problem is mainly based on the strong uniqueness property and Runge approximation property for anisotropic fractional Schr\"odinger operators. 

\medskip

  \smallskip

\end{abstract}

%\keywords{inverse scattering, Maxwell equations, polyhedral scatterers, uniqueness, stability, $N-1$ far-field measurements}
%\classification{35R}

\maketitle

\tableofcontents

\section{Introduction}

\subsection{Mathematical setup and statement of the main results}

Let $\mathrm{Sym}(n\times n)$ signify the space of real-valued $n\times n$ symmetric matrices for $n\geq 2$. Let $A(x)=(a_{ij}(x))_{i,j=1}^n\in\mathrm{Sym}(n\times n)$, $x\in\mathbb{R}^n$. Throughout, it is assumed that $A$ satisfies the following uniform ellipticity condition for some $\gamma\in (0,1)$,  
\begin{equation}\label{ellipticity and symmetry condition} 
\gamma|\xi|^{2}\leq\sum_{i,j=1}^{n}a_{ij}(x)\xi_{i}\xi_{j}\leq\gamma^{-1}|\xi|^{2} \ \mbox{ for all } \xi, x\in\mathbb{R}^{n},
\end{equation}
and $A(x)$ is $C^\infty$-smooth for any $x\in \mathbb R^n$.
Define $\mathscr{L}_A$ to be the following PDO (partial differential operator),
\begin{equation}\notag
\mathscr{L}_A:=-\nabla\cdot(A(x)\nabla),
\end{equation}
Let $s\in (0, 1)$ be a constant and introduce the following nonlocal PDO,
\begin{equation}\notag
\mathscr{L}_A^{s}=(-\nabla\cdot(A(x)\nabla)^{s},
\end{equation} 
whose rigorous definition shall be given in Section 2.

Let $\Omega$ and $D$ be two bounded open sets in $\mathbb{R}^n$ such that $D\Subset\Omega$ and, $\mathbb{R}^n\backslash\overline{\Omega}$ and $\Omega\backslash\overline{D}$ are connected. Let $q\in L^\infty(\Omega\backslash\overline{D})$ be a real-valued function. Physically speaking, $q$ and $D$, respectively, represent a potential and an embedded impenetrable obstacle inside the potential. 

Consider the following nonlocal problem associated with $q$ and $D$, 
\begin{equation}
\begin{cases}
\mathscr{L}_A^{s}u+qu=0 & \mbox{ in }\ \ \Omega\backslash\overline{D},\\
\mathcal{B}u=0 & \mbox{ in }\ \ D,\\
u=g & \mbox{ in }\ \ \Omega_{e}:=\mathbb{R}^{n}\backslash\overline{\Omega},
\end{cases}\label{Fractional Schrodinger equation}
\end{equation}
where $u\in H^{s}(\mathbb{R}^{n})$ is a weak solution of \eqref{Fractional Schrodinger equation}
with $g\in H^{s}(\mathbb R^n)$ being an exterior Dirichlet data. In \eqref{Fractional Schrodinger equation}, $\mathcal{B} u:=u$ if $D$ is a {\it soft} obstacle, and $\mathcal{B}u:=\mathscr{L}_A^s u$ if $D$ is a {\it hard} obstacle. It is known
that \eqref{Fractional Schrodinger equation} is uniquely solvable
if $\{0\}$ is not an eigenvalue of the operator
$\mathscr{L}_A^{s}+q$, in the following sense
\begin{equation}
\begin{cases}
\mbox{if } w\in H^{s}(\mathbb{R}^{n})\mbox{ solves }(\mathscr{L}_A^{s}+q)w=0\mbox{ in }\Omega\backslash\overline{D},\\
w=0\mbox{ in }\Omega_{e},\mbox{ and } \mathcal Bu=0 \text{ in }D,\\
\mbox{then } w\equiv 0.
\end{cases}\label{eq:eigenvalue condition}
\end{equation}
Throughout, we assume that $\{0\}$ is not an eigenvalue of $\mathscr{L}_A^{s}+q$, and hence \eqref{Fractional Schrodinger equation} is well-posed. In particular, one has the following well-defined Dirichlet-to-Neumann (DtN) map associated to the nonlocal problem (\ref{Fractional Schrodinger equation}),
\[
\Lambda_{D,q}:\mathbb{H}\to\mathbb{H}^{*},
\]
and $\Lambda_{D,q}$ is formally given by 
\begin{equation}\notag
\Lambda_{D,q}\psi:=\mathscr{L}_A^{s}u_{\psi}|_{\Omega_{e}},\label{DtN map}
\end{equation}
where $u_{\psi}$ is the unique solution to \eqref{Fractional Schrodinger equation}
with $u_{\psi}=\psi$ in $\Omega_{e}$. In the subsequent section, we shall introduce more details of the abstract Banach spaces $\mathbb{H}$ and $\mathbb{H}^{*}$. We regard the DtN map $\Lambda_{D,q}$ as the exterior measurement for our inverse problem study. In this article, we are mainly concerned with the recovery of the embedded obstalce $D\Subset\Omega$ and the surrounding potential $q(x)\in L^{\infty}(\Omega\backslash\overline{D})$
by using the exterior DtN map of $(\mathscr{L}_A^{s}+q)u=0$ in $\Omega\backslash\overline{D}$.

For the inverse problem described above, our main results can be stated as follows. 
\begin{thm}
	\label{Main Theorem 1} For $n\ge2$, let $\Omega\subset\mathbb{R}^{n}$
	be an open bounded set, $D_{1},D_{2}\Subset\Omega$ be two open subsets of $\Omega$ 
	and $\mathcal{O}_1,\mathcal O_2\subset\Omega_{e}$ be nonempty open sets. Suppose $D_j$ and
	$q_{j}\in L^{\infty}(\Omega\backslash\overline{D_{j}})$ satisfy
	the eigenvalue condition \eqref{eq:eigenvalue condition}, $j=1,2$. Let $\Lambda_{D_j,q_j} $be the
	DtN maps for the nonlocal equations $\left(\mathscr{L}_A^{s}+q_{j}\right)u_{j}=0\mbox{ in }\Omega\backslash\overline{D_{j}}$
	with $u_{j}=0$ in $D_{j}$ for $j=1,2$, then the following statements hold.
	
	1. For any given $g\in C^\infty_c(\mathcal O_1)$ with $g\not \equiv 0$ in $\mathcal O_1$, if
	\[
	\Lambda_{D_{1},q_{1}}g|_{\mathcal{O}_2}=\Lambda_{D_{2},q_{2}}g|_{\mathcal{O}_2},
	\]
	then one has $D_{1}=D_{2}$.
	
	2. Furthermore, if
	\[
	\Lambda_{D_{1},q_{1}}g|_{\mathcal{O}_2}=\Lambda_{D_{2},q_{2}}g|_{\mathcal{O}_2}\mbox{ for all }g\in C_{c}^{\infty}(\mathcal{O}_1),
	\]
	then one has $q_{1}=q_{2}$ in $\Omega\backslash\overline{D}$, where $D:=D_{j}$
	for $j=1,2$.\end{thm}
Moreover, if we further assume $q(x)\neq 0$ for any $x\in \Omega$, then we have the following unique recovery result for the sound hard case.
\begin{thm}
	\label{Main Theorem 2} Let $\Omega, \mathcal{O}_j$ and $D_j, q_j$, $j=1,2$, be the same as those described in Theorem~\ref{Main Theorem 1}. Let $\Lambda_{D_j,q_j} $be the
	DtN maps for the nonlocal equations $\left(\mathscr{L}_A^{s}+q_{j}\right)u_{j}=0\mbox{ in }\Omega\backslash\overline{D_{j}}$
	with $\mathscr L_A^s u_{j}=0$ in $D_{j}$ for $j=1,2$, then the following statements hold.
	
	1. We further assume that $q_j(x)\neq 0$ for $x\in \Omega$ and $j=1,2$. For any given $g\in C^\infty_c(\mathcal O_1)$ with $g\not \equiv 0$ in $\mathcal O$, if  
	\[
	\Lambda_{D_{1},q_{1}}g|_{\mathcal{O}_2}=\Lambda_{D_{2},q_{2}}g|_{\mathcal{O}_2},
	\]
	then one has $D_{1}=D_{2}$.
	
	2. Furthermore, if
	\[
	\Lambda_{D_{1},q_{1}}g|_{\mathcal{O}_2}=\Lambda_{D_{2},q_{2}}g|_{\mathcal{O}_2}\mbox{ for all }g\in C_{c}^{\infty}(\mathcal{O}_1),
	\]
	then one has $q_{1}=q_{2}$ in $\Omega\backslash\overline{D}$, where $D:=D_{j}$
	for $j=1,2$.
	\end{thm}

By the first statement in Theorem~\ref{Main Theorem 1} or \ref{Main Theorem 2}, a single pair of nontrivial Cauchy data $(g, \Lambda_{D, q}g)$ is sufficient to uniquely recover the embedded soft or hard obstacle $D$, independent of the surrounding potential $q$. It is also noted that no restrictive regularity assumption is required on the obstacle $D$. If multiple measurements are used, then both the embedded obstacle and the surrounding potential can be uniquely recovered. We can further show that the recovery of the embedded obstacle can be achieved without knowing it is soft or hard. Indeed, by virtue of Theorems~\ref{Main Theorem 1} and \ref{Main Theorem 2}, it suffices for us to establish the following result. 
	
\begin{thm}
	\label{Main Theorem 3} Let $\Omega, \mathcal{O}_j$ and $D_j, q_j$, $j=1,2$, be the same as those described in Theorem~\ref{Main Theorem 1}. Let $\Lambda_{D_j,q_j} $be the
	DtN maps for the nonlocal equations $\left(\mathscr{L}_A^{s}+q_{j}\right)u_{j}=0\mbox{ in }\Omega\backslash\overline{D_{j}}$
	with either
	\begin{align}
	\notag	u_1=0\text{ in }D_1 \text{ and }\mathscr L_A^s u_2 =0 \text{ in }D_2
	\end{align}
	or
	\begin{align}
	\notag \mathscr L_A^s u_1 =0 \text{ in }D_1\text{ and }u_2=0\text{ in }D_2,
	\end{align}
	then the following statements hold.
	
	1. We further assume that $q_j(x)\neq 0$ for $x\in \Omega$ and $j=1,2$. For any given $g\in C^\infty_c(\mathcal O_1)$ with $g\not \equiv 0$ in $\mathcal O_1$, if  
	\[
	\Lambda_{D_{1},q_{1}}g|_{\mathcal{O}_2}=\Lambda_{D_{2},q_{2}}g|_{\mathcal{O}_2},
	\]
	then one has $D_{1}=D_{2}$.
	
	2. Furthermore, if
	\[
	\Lambda_{D_{1},q_{1}}g|_{\mathcal{O}_2}=\Lambda_{D_{2},q_{2}}g|_{\mathcal{O}_2}\mbox{ for all }g\in C_{c}^{\infty}(\mathcal{O}_1),
	\]
	then one has $q_{1}=q_{2}$ in $\Omega\backslash\overline{D}$, where $D:=D_{j}$
	for $j=1,2$.
	
	\end{thm}

\subsection{Discussion and historical remarks }
The study of nonlocal inverse problems has received significant attention in the literature in recent years. The Calder\'{o}n
problem for the fractional Schr\"{o}dinger equation was first solved by
Ghosh, Salo and Uhlmann \cite{ghosh2016calder}. Based on the similar
idea, \cite{ghosh2017calderon} and \cite{lai2017global} generalized
the results to the nonlocal variable case and nonlocal semilinear
case, respectively. Note that the global uniqueness results hold for
these nonlocal cases for any space dimension $n\geq2$. The proof
of the Calder\'{o}n problem strongly relies on the \textit{strong uniqueness
	property}, and we refer readers to \cite[Theorem 1.2]{ghosh2016calder}
for the fractional Laplace $(-\Delta)^{s}$ and \cite[Theorem 1.2]{ghosh2017calderon}
for the nonlocal variable operator $\mathscr{L}_A^{s}$. The strong uniqueness
means that: for $s\in(0,1)$, $u\in H^s(\mathbb R^n)$, if $u=\mathscr{L}_A^{s}u=0$ in an arbitrary
open set in $\mathbb{R}^{n}$, then $u\equiv0$ in $\mathbb{R}^{n}$
for any $n\geq2$. Based on the strong uniqueness property, one can
obtain the nonlocal Runge approximation property, which states that any
$L^{2}$ function can be approximated by a sequence of the solutions of
$(\mathscr{L}_A^{s}+q)u=0$.

Recently, R\"uland and Salo \cite{ruland2017fractional} studied
the fractional Calder\'{o}n problem under lower regularity conditions
and established the stability results for the determination of unknown potentials. They \cite{ruland2017exponential} proved the optimal logarithmic stability
for the corresponding inverse problem associated with the fractional Schr\"{o}dinger equation. In \cite{harrach2017monotonicity},
the authors characterized an if-and-only-if relationship between two positive
potentials and their associated DtN maps of the fractional Schr\"{o}dinger
equation. Harrach and Lin \cite{harrach2017monotonicity} also provided
a reconstruction algorithm of an unknown inclusion based on the monotonicity
method. The nonlocal inverse problems reveal some novel and distinctive features compared to their local counterparts. For the current study of simultaneously recovering unknown potentials with possibly embedded impenetrable obstacles, we next also provide some interesting discussions and observations compared to its local counterpart.

When $s=1$, \eqref{Fractional Schrodinger equation} becomes a local problem and in such a case, one should replace the nonlocal condition $\mathcal{B} u=0$ in $D$ by $\widetilde{\mathcal{B}}(u)=0$ on $\partial D$, where $\widetilde{\mathcal{B}}u=u$ if $D$ is a soft obstacle and $\widetilde{\mathcal{B}}u=\nu^T\cdot A\cdot\nabla u$ if $D$ is a hard obstacle, with $\nu$ signifying the exterior unit normal vector to $\partial D$. The corresponding local DtN map can be readily defined on $\partial D$, which we still denote by $\Lambda_{D, q}$. The local inverse problem of determining $D$ by $\Lambda_{D, q}$ is usually referred to as the {\it obstacle problem}. The obstacle problem by a single measurement, namely determining $D$ by a single pair of Cauchy data $(\psi, \Lambda_{D, q}\psi)$ is a well-known and longstanding problem in the inverse scattering theory, which is also known as the Schiffer's problem, particularly for the case $A=I$ and $q=1$ \cite{CK,Isa2,LZsurvey}. There has been extensive study in the literature and significant progress has been achieved in recent years on the Schiffer's problem for the case with general polyhedral obstacles; see \cite{AR,CY,Liu1,Liu2} and \cite{LMRX,Ron1,Ron2} and the references therein, respectively, for related uniqueness and stability studies. Under the restrictive assumption that $\partial D$ is everywhere non-analytic, the Schiffer's problem was solved in \cite{HNS}. However, for the case with general obstacles, the Schiffer's problem still remains open in the literature. According to Theorems 1.1--1.3, the nonlocal Schiffer's problem has been completely solved in our study. Hence, it would be much interesting to study the connection of the nonlocal and local Schiffer's problems. This might be partly seen by taking the limit $s\rightarrow 1^{-}$. The simultaneous recovery of an embedded obstacle and an unknown surrounding potential is also a longstanding problem in the literature and closely related to the so-called partial data Calder\'on problem \cite{CSU,IUY}. The existing unique recovery results were established based on knowing the embedded obstacle to recover the unknown potential \cite{IUY}, or knowing the surrounding potential to recover the embedded obstacle \cite{KL,KP,LiuZhang,LiuZhaoZou,O}, or using multiple spectral data to recover both of them \cite{LiuLiu}. 

The rest of the paper is structured as follows. In Section \ref{Section 2},
we provide rigorous mathematical formulations of the nonlocal elliptic
operator $\mathscr{L}_A^{s}$ and fractional Sobolev spaces. In Section
\ref{Section 3}, we study the well-posedness and the associated
DtN map for $\mathscr{L}_A^{s}+q$ with an embedded obstacle.
In Section \ref{Section 4}, we prove the uniqueness in determining the obstacle $D$ in $\Omega$ by using a single exterior measurement. In Section \ref{Section 5}, we prove the global uniqueness in recovering the surrounding potential $q$. Combine with Section \ref{Section 4} and \ref{Section 5}, then we prove Theorem \ref{Main Theorem 1}--\ref{Main Theorem 3}.

\section{Preliminary knowledge on $\mathscr{L}_A^{s}$ \label{Section 2}}

In this section, we present some preliminary knowledge on the nonlocal PDO $\mathscr{L}_A^{s}$ that shall be needed in our inverse problem study. We begin with the definition of fractional Sobolev spaces.

\subsection{Fractional Sobolev spaces}

For $0<s<1$, the fractional Sobolev space is denoted by $H^{s}(\mathbb{R}^{n}):=W^{s,2}(\mathbb{R}^{n})$,
which is the standard $L^{2}$ based Sobolev space with the norm 
\begin{equation}\notag
\|u\|_{H^{s}(\mathbb{R}^{n})}^{2}=\|u\|_{L^{2}(\mathbb{R}^{n})}^{2}+\|(-\Delta)^{s/2}u\|_{L^{2}(\mathbb{R}^{n})}^{2}.\label{eq:H^s norm}
\end{equation}
The semi-norm $\|(-\Delta)^{s/2}u\|_{L^{2}(\mathbb{R}^{n})}^{2}$
can also be expressed as 
\[
\|(-\Delta)^{s/2}u\|_{L^{2}(\mathbb{R}^{n})}^{2}=\left((-\Delta)^{s}u,u\right)_{\mathbb{R}^{n}},
\]
where 
\begin{equation}
\notag(-\Delta)^su = c_{n,s}\, \text{P.V.}\int_{\mathbb{R}^n}\frac{u(x)-u(y)}{|x-y|^{n+2s}}\, dy\end{equation}
is the standard fractional Laplacian with the constant  \begin{equation}\notag c_{n,s}= \frac{\Gamma(\frac{n}{2}+s)}{|\Gamma(-s)|}\,\frac{4^s}{\pi^{n/2}}\end{equation}
and $\text{P.V.}$ denotes the standard principal value operator (see \cite{di2012hitchhiks} for detailed description).

Next, given any open set $U$ of $\mathbb{R}^{n}$ and $\eta\in\mathbb{R}$,
we consider the following Sobolev spaces, 
\begin{align*}
H^{\eta}(U) & :=\{u|_{U};\,u\in H^{\eta}(\mathbb{R}^{n})\},\\
\widetilde{H}^{\eta}(U) & :=\text{closure of \ensuremath{C_{c}^{\infty}(U)} in \ensuremath{H^{\eta}(\mathbb{R}^{n})}},\\
H_{0}^{\eta}(U) & :=\text{closure of \ensuremath{C_{c}^{\infty}(U)} in \ensuremath{H^{\eta}(U)}},
\end{align*}
and 
\[
H_{\overline{U}}^{\eta}:=\{u\in H^{\eta}(\mathbb{R}^{n});\,\mathrm{supp}(u)\subset\overline{U}\}.
\]
The Sobolev space $H^{\eta}(U)$ is complete under the graph norm
\[
\|u\|_{H^{\eta}(U)}:=\inf\left\{ \|v\|_{H^{\eta}(\mathbb{R}^{n})};v\in H^{\eta}(\mathbb{R}^{n})\mbox{ and }v|_{U}=u\right\} .
\]
It is known that $\widetilde{H}^{\eta}(U)\subseteq H_{0}^{\eta}(U)$,
and $H_{\overline{U}}^{\eta}$ is a closed subspace of $H^{\eta}(\mathbb{R}^{n})$.
For more detailed discussion of the fractional Sobolev spaces, we
refer to \cite{di2012hitchhiks,mclean2000strongly}.

\subsection{Definition of $\mathscr{L}_A^{s}$}

Let us get into the rigorous mathematical formulation of the problem we study here.
Let us begin with the definition of the nonlocal operator $\mathscr{L}_A^{s}$,
$s\in(0,1)$ via the spectral characterization of $\mathscr{L}_A$.
Suppose that $\mathscr{L}_A$ is a linear second order self-adjoint
elliptic operator, which is densely defined on $L^{2}(\mathbb{R}^{n})$
for $n\geq2$. There is a unique resolution $E$ of the identity,
supported on the spectrum of $\mathscr{L}_A$ which is a subset of $[0,\infty)$,
such that 
\[
\mathscr{L}_A=\int_{0}^{\infty}\lambda dE(\lambda)
\]
i.e., 
\[
\langle\mathscr{L}_Af,g\rangle_{L^{2}(\mathbb{R}^{n})}=\int_{0}^{\infty}\lambda dE_{f,g}(\lambda),\ f\in\mbox{Dom}(\mathscr{L}_A),g\in L^{2}(\mathbb{R}^{n}),
\]
where $dE_{f,g}(\lambda)$ is a regular Borel complex measure of bounded
variation concentrated on the spectrum of $\mathscr{L}_A$, such that $d|E_{f,g}|(0,\infty)\leq\|f\|_{L^{2}(\mathbb{R}^{n})}\|g\|_{L^{2}(\mathbb{R}^{n})}$.

If $\phi(\lambda)$ is a real measurable function defined on $[0,\infty)$,
then the operator $\phi(\mathscr{L}_A)$ is given formally by 
\[
\phi(\mathscr{L}_A)=\int_{0}^{\infty}\phi(\lambda)dE(\lambda).
\]
That is, $\phi(\mathscr{L}_A)$ is the operator with domain 
\[
\mbox{Dom}(\phi(\mathscr{L}_A))=\left\{ f\in L^{2}(\mathbb{R}^{n}):\int_{0}^{\infty}\|\phi(\lambda)\|^{2}dE_{f,f}(\lambda)<\infty\right\} ,
\]
defined by 
\[
\left\langle \phi(\mathscr{L}_A)f,g\right\rangle _{L^{2}(\mathbb{R}^{n})}=\left\langle \int_{0}^{\infty}\phi(\lambda)dE(\lambda)f,g\right\rangle _{L^{2}(\mathbb{R}^{n})}=\int_{0}^{\infty}\phi(\lambda)dE_{f,g}(\lambda).
\]
Following that we define the nonlocal elliptic operators $\mathscr{L}_A^{s}$,
$s\in(0,1)$ with domain $\mbox{Dom}(\mathscr{L}_A^{s})\subset\mbox{Dom}(\mathscr{L}_A)$,
\begin{equation}\notag
\mathscr{L}_A^{s}=\int_{0}^{\infty}\lambda^{s}~dE(\lambda)=\frac{1}{\Gamma(-s)}\int_{0}^{\infty}\left(e^{-t\mathscr{L}_A}-\mbox{Id}\right)~\frac{dt}{t^{1+s}},\label{eq:1111}
\end{equation}
where $\Gamma(s)$ is the standard Gamma function and $\Gamma(-s)=\dfrac{\Gamma(1-s)}{-s}<0$
for $s\in(0,1)$. Here $e^{-t\mathscr{L}_A}$ $(t\geq0)$ is the heat-diffusion
semigroup generated by $\mathscr{L}_A$ with domain $L^{2}(\mathbb{R}^{n})$
\begin{equation}\notag
e^{-t\mathscr{L}_A}=\int_{0}^{\infty}e^{-t\lambda}~dE(\lambda),\label{eq:heat-semigroup}
\end{equation}
which enjoys the contraction property in $L^{2}(\mathbb{R}^{n})$
as $\|e^{-t\mathscr{L}_A}f\|\leq\|f\|_{L^{2}(\mathbb{R}^{n})}$. Meanwhile,
for $w\in H^{s}(\mathbb{R}^{n})$, we have 
\begin{equation}\notag
\mathscr{L}_A^{s}w=\frac{1}{\Gamma(-s)}\int_{0}^{\infty}\left(e^{-t\mathscr{L}_A}w(x)-w(x)\right)\dfrac{dt}{t^{1+s}}.\label{eq:heat representation for L^s}
\end{equation}
For more detailed discussions, we refer readers to \cite{pazy2012semigroups,rudin1991functional,stinga2010extension}.

In fact, the heat-diffusion semigroup admits a nonnegative symmetric heat kernel $\mathscr{W}_{t}(x,z)$, $t>0$, $x,z\in \Omega$ by integration, that is for any $f\in L^2(\mathbb{R}^n)$
\begin{equation*}
e^{-t\mathscr{L}_A}f(x)=\int_{\Omega}\mathscr{W}_{t}(x,z)f(\eta)d\eta(z)
\end{equation*}
and for any $v,w\in H^{s}(\mathbb{R}^{n})$, 
\[
(e^{-t\mathscr{L}_A}v,w)_{\mathbb{R}^{n}}=\int_{\mathbb{R}^{n}}\int_{\mathbb{R}^{n}}\mathscr{W}_{t}(x,z)v(z)w(x)dzdx=(v,e^{-t\mathscr{L}_A}w)_{\mathbb{R}^{n}},\mbox{ }t\geq0.
\]
Define 
\begin{equation}\notag
\mathscr{K}_{s}(x,z)=\dfrac{1}{2|\Gamma(-s)|}\int_{0,}^{\infty}\mathscr{W}_{t}(x,z)\dfrac{dt}{t^{1+s}},\label{eq:kernel}
\end{equation}
which gives the kernel of the heat semi-group $e^{-t\mathscr{L}_A}$
and utilizes \cite[Theorem 2.4]{caffarelli2016fractional}, 
\begin{equation}
(\mathscr{L}_A^{s}v,w)_{\mathbb{R}^{n}}=\int_{\mathbb{R}^{n}}\int_{\mathbb{R}^{n}}(v(x)-v(z))(w(x)-w(z))\mathscr{K}_{s}(x,z)dxdz,\label{eq:integral represent for nonlocal}
\end{equation}
where we used the fact that $A(x)$ is a bounded matrix-valued function
defined  in $\mathbb{R}^{n}$ satisfying \eqref{ellipticity and symmetry condition}. In addition, the kernel $\mathscr{K}_{s}$
possesses the following pointwise estimate (see \cite[Theorem 2.4]{caffarelli2016fractional}
again) 
\begin{equation}
\dfrac{c_{1}}{|x-z|^{n+2s}}\leq\mathscr{K}_{s}(x,z)\leq\dfrac{c_{2}}{|x-z|^{n+2s}},\label{eq:pointwise estimate for kernel K}
\end{equation}
for some constants $c_{1}$, $c_{2}>0$ depending on $A$,
$n$ and $s$ and $\mathscr{K}_{s}(x,z)=\mathscr{K}_{s}(z,x)$ for
all $x,z\in\mathbb{R}^{n}$. We also refer readers to \cite{ghosh2017calderon}
for further discussions of the nonlocal operator $\mathscr{L}_A^{s}$.

\section{Nonlocal problems with embedded obstacles and the surrounding potentials \label{Section 3}}

In this section, we give the mathematical formulations of our nonlocal
problems.

\subsection{Well-Posedness}

In the subsequent discussions, we always set $\Omega\subseteq\mathbb{R}^{n}$
to be a bounded open set and $D\Subset\Omega$ to be an open subset, $q$
to be a potential in $L^{\infty}(\Omega\backslash\overline{D})$ and
$s\in(0,1)$ to be a constant. Consider the nonlocal Dirichlet problem
\begin{equation}
\begin{cases}
(\mathscr{L}_A^{s}+q)u=f & \mbox{ in }\Omega\backslash\overline{D},\\
\mathcal Bu=0& \mbox{ in }D,\\
u=g & \mbox{ in }\Omega_{e}.
\end{cases}\label{eq:Nonlocal Dirichlet problem}
\end{equation}
Define the bilinear form $\mathbb{B}_{q}(\cdot,\cdot)$ by 
\begin{eqnarray}
\mathbb{B}_{q}(v,w): & = & \int_{\mathbb{R}^{n}}\int_{\mathbb{R}^{n}}(v(x)-v(z))(w(x)-w(z))\mathscr{K}_{s}(x,z)dxdz\notag\\
&  & +\int_{\Omega\backslash\overline{D}}q(x)v(x)w(x)\,dx,\label{Bilinera form}
\end{eqnarray}
for any $v,w\in H^{s}(\mathbb{R}^{n})$. Combining \eqref{eq:integral represent for nonlocal}
and \eqref{Bilinera form}, we have that 
\[
\mathbb{B}_{q}(v,w)=\int_{\mathbb{R}^{n}}(\mathscr{L}_A^{s}v)w\,dx+\int_{\Omega\backslash\overline{D}}qvw\:dx.
\]
Then by using the standard variational formula, the weak solution
can be defined by 
\begin{defn}
	(Weak solution) Let $\Omega$ be a bounded open set in $\mathbb{R}^{n}$.
	Given $f\in L^{2}(\Omega\backslash\overline{D})$ and $g\in H^{s}(\mathbb{R}^{n})$,
	we call that $u\in H^{s}(\mathbb{R}^{n})$ is a (weak) solution of
	\eqref{eq:Nonlocal Dirichlet problem} provided that $\widetilde{u}_{g}:=u-g\in\widetilde{H}^{s}(\Omega)$
	and 
	\begin{equation}
	\mathbb{B}_{q}(u,\phi)=\int_{\Omega\backslash\overline{D}}f\phi dx\quad\mbox{for any \ensuremath{\phi\in C_{0}^{\infty}(\Omega\backslash\overline{D})}},\label{eq:weak solution via bilinear}
	\end{equation}
	with $u-g\in\widetilde{H}^{s}(\Omega)$ or equivalently 
	\begin{equation}\notag
	\mathbb{B}_{q}(\widetilde{u}_{g},\phi)=\int_{\Omega\backslash\overline{D}}\left(f-(\mathscr{L}_A^{s}+q)g\right)\phi dx\quad\mbox{for any \ensuremath{\phi\in C_{0}^{\infty}(\Omega\backslash\overline{D})}}.\label{eq:weak solution via bilinear1}
	\end{equation}
	
\end{defn}
Next, we have the following well-posedness. 
\begin{lem}
	\label{lemma:equivalence} Let $q\in L^{\infty}(\Omega\backslash\overline{D})$
	and $f\in L^{2}(\Omega\backslash\overline{D})$. $u\in H^{s}(\mathbb{R}^{n})$
	solves 
	\[
	\mathscr{L}_A^{s}u+qu=f\quad\text{ in }\Omega\backslash\overline{D},
	\]
	(in the sense of distributions) if and only if $u\in H^{s}(\mathbb{R}^{n})$
	satisfies 
	\[
	\mathbb{B}_{q}(u,w)=\int_{\Omega\setminus\overline{D}}fw {dx}\quad\text{ for all }w\in\widetilde{H}^{s}(\Omega\backslash\overline{D}).
	\]
	Moreover, when $q$ satisfies the eigenvalue condition \eqref{eq:eigenvalue condition},
	we have the stability estimate 
	\begin{equation}
	\|u\|_{H^{s}(\mathbb{R}^{n})}\leq C\left(\|f\|_{L^{2}(\Omega\backslash\overline{D})}+\|g\|_{H^{s}(\mathbb{R}^{n})}\right),\label{eq:well posed estimate}
	\end{equation}
	where $C>0$ is a constant independent of $f$ and $g$.\end{lem}
\begin{proof}
	A straightforward computation shows that 
	\begin{align*}
	& \int_{\Omega\backslash\overline{D}}\left(\mathscr{L}_A^{s}u+qu-f\right)wdx\\
	= & \int_{\mathbb{R}^{n}}\int_{\mathbb{R}^{n}}(u(x)-u(z))(w(x)-w(z))\mathscr{K}_{s}(x,z)dxdz\\
	&+\int_{\Omega\backslash\overline{D}}quwdx-\int_{\Omega\backslash\overline{D}}fwdx
	\end{align*}
	for all $w\in C_{c}^{\infty}(\Omega\backslash\overline{D})$. It is
	easy to see that the bilinear form $\mathbb{B}_{q}(\cdot,\cdot)$
	is bounded, coercive and continuous by using the pointwise estimate
	\eqref{eq:pointwise estimate for kernel K} of the kernel $\mathscr{K}_{s}(x,z)$,
	then the stability estimate \eqref{eq:well posed estimate} follows
	from the standard Lax-Milgram theorem (a similar proof has been addressed
	in \cite[Section 3]{ghosh2017calderon}). This completes the proof.\end{proof}
\begin{lem}
	\label{rem:uIndependentg} The solution $u\in H^{s}(\mathbb{R}^{n})$
	of \eqref{eq:Nonlocal Dirichlet problem} is independent of the value
	of $g\in H^{s}(\mathbb{R}^{n})$ in $\Omega$, and it only relies on $g|_{\Omega_{e}}$.\end{lem}
\begin{proof}
	The proof is similar to that of \cite[Proposition 3.4]{ghosh2017calderon}
	and we skip it. 
\end{proof}
Via Lemma \ref{rem:uIndependentg}, we can consider the nonlocal problem
\eqref{Fractional Schrodinger equation} with Dirichlet data in an
abstract quotient space 
\begin{equation}
\mathbb{H}:=H^{s}(\mathbb{R}^{n})/\widetilde{H}^{s}(\Omega).\label{eq:Xquotient}
\end{equation}
We also refer readers to \cite{ghosh2017calderon,ghosh2016calder}
for more detailed discussions. Since the solution $u\in H^{s}(\mathbb{R}^{n})$
of \eqref{eq:Nonlocal Dirichlet problem} only depends on the exterior
value, in order to simplify notations, we shall consider the Dirichlet data
$g$ in the quotient space $\mathbb{H}$ in the subsequent study.

\subsection{The DtN map}

We define the associated DtN map for $\mathscr{L}_A^{s}+q$ via
the bilinear form $\mathbb{B}_{q}$ in \eqref{eq:weak solution via bilinear}. 
\begin{prop}
	\label{prop:DNmap} (DtN map) For $n\geq2$, let $\Omega\subset\mathbb{R}^{n}$
	be a bounded open set and $D\Subset\Omega$ be an obstacle. Let $0<s<1$
	and $q\in L^{\infty}(\Omega\backslash\overline{D})$ satisfy \eqref{eq:eigenvalue condition}.
	Let $\mathbb{H}$ be the abstract space given in \eqref{eq:Xquotient}.
	Define 
	\begin{equation}\notag
	\left\langle \Lambda_{D,q}g,h\right\rangle _{\mathbb{H}^{*}\times\mathbb{H}}:=\mathbb{B}_{q}(u_{g},h),\quad g,h\in\mathbb{H},\label{eq:equvalent integration by parts}
	\end{equation}
	where $u_{g}\in H^{s}(\mathbb{R}^{n})$ is the solution of \eqref{Fractional Schrodinger equation}
	with the exterior Dirichlet data $g$. Then $\Lambda_{D,q}:\mathbb{H}\to\mathbb{H}^{*}$
	is a bounded linear map. Moreover, we have the following symmetry
	property for $\Lambda_{D,q}$, 
	\begin{equation}\notag
	\left\langle \Lambda_{D,q}g,h\right\rangle _{\mathbb{H}^{*}\times\mathbb{H}}=\left\langle \Lambda_{D,q}h,g\right\rangle _{\mathbb{H}^{*}\times\mathbb{H}},\quad g,h\in\mathbb{H}.\label{eq:adjoint operator}
	\end{equation}
\end{prop}
\begin{proof}
	Combining with Lemma \ref{rem:uIndependentg}, the proof is similar
	to \cite[Proposition 3.5]{ghosh2017calderon}, so we skip it here.\end{proof}
\begin{rem}
	For any $\widehat{h}\in H^{s}(\mathbb{R}^{n})$, a direct calculation
	shows that 
	\begin{align}
	(\Lambda_{D,q}g,h)_{\mathbb{H}^{*}\times\mathbb{H}} & =\mathbb{B}_{q}(u_{g},\widehat{h})\nonumber \\
	& =\int_{\mathbb{R}^{n}}\widehat{h}(\mathscr{L}_A^{s}u_{g})dx+\int_{\Omega}qu_{g}\widehat{h}dx\nonumber \\
	& =\int_{\Omega_{e}}\widehat{h}(\mathscr{L}_A^{s}u_{g})dx\nonumber \\
	& =\int_{\Omega_{e}}h(\mathscr{L}_A^{s}u_{g})\,dx.\label{t15}
	\end{align}
	Then from \eqref{t15}, we have 
	\[
	(\Lambda_{q}g,h)_{\mathbb{H}^{*}\times\mathbb{H}}=\int_{\Omega_{e}}h(\mathscr{L}_A^{s}u_{g})\,dx,\mbox{ for any }h\in\mathbb{H},
	\]
	which implies that 
	\begin{equation}
	\Lambda_{q}g=\left.\mathscr{L}_A^{s}u_{g}\right|_{\Omega_{e}}.\notag
	\end{equation}
	
\end{rem}
The integral identity allows us to solve the nonlocal type inverse
problem as a direct consequence of Proposition \ref{prop:DNmap}.
It can be stated as follows.
\begin{lem}
	\label{prop:Integral identity} (Integral identity) For $n\geq2$,
	let $\Omega\subset\mathbb{R}^{n}$ be a bounded open set and $D\Subset\Omega$
	be a obstacle. Let $s\in(0,1)$ and $q\in L^{\infty}(\Omega\backslash\overline{D})$
	satisfy \eqref{eq:eigenvalue condition}. For any $g_{1},g_{2}\in\mathbb{H}$,
	one has 
	\begin{equation}\notag
	\int_{\Omega_e}(\Lambda_{D,q_{1}}g_{1}-\Lambda_{D,q_{2}}g_{1})g_{2}\,dx=\int_{\mathbb{R}^{n}}(q_{1}-q_{2})r_{\Omega\backslash\overline{D}}u_{1}r_{\Omega\backslash\overline{D}}u_{2}\,dx
	\end{equation}
	where $u_{j}\in H^{s}(\mathbb{R}^{n})$ solves $(\mathscr{L}_A^{s}+q_{j})u_{j}=0$
	in $\Omega\backslash\overline{D}$ with $\left.u_{j}\right|_{\Omega_{e}}=g_{j}$
	for $j=1,2$, and $r_{\Omega\backslash\overline{D}}u$ refers to the
	restriction of $u$ to $\Omega\backslash\overline{D}$.\end{lem}
\begin{proof}
	The proof is similar to \cite[Lemma 2.5]{ghosh2016calder}. 
\end{proof}

\section{Recovery of the obstacle $D$\label{Section 4}}

In this section, we show that the obstacle $D$ can be uniquely recovered by a single measurement. The following strong uniqueness property shall be needed. 

\begin{prop}
	\cite[Theorem 1.2]{ghosh2017calderon} \label{Prop Strong unique}For
	$n\geq2$ and $0<s<1$. If $u\in H^{s}(\mathbb{R}^{n})$ satisfies
	$u=\mathscr{L}_A^{s}u=0$ in any nonempty open set $U\subset\mathbb{R}^{n}$,
	then $u\equiv0$ in $\mathbb{R}^{n}$. 
\end{prop}

Now we can prove the first statement of Theorem \ref{Main Theorem 1}. 

\begin{thm}
	\label{thm:uniqueness of D} Let $\Omega$ be a bounded open set in
	$\mathbb{R}^{n}$, $D_{1},D_{2}\Subset\Omega$ be two open subsets and
	$\mathcal{O}_1,\mathcal O_2\subset\Omega_{e}$ be arbitrary nonempty open sets.
	Let $q_j\in L^{\infty}(\Omega\backslash\overline{D})$ satisfy \eqref{eq:eigenvalue condition} and $u_{j}\in H^{s}(\mathbb{R}^{n})$ be the unique (weak) solution
	of 
	\[
	\begin{cases}
	\mathscr{L}_A^{s}u_{j}+q_{j}u_{j}=0 & \mbox{ in }\Omega\backslash\overline{D_{j}},\\
	\mathcal Bu_{j}=0  & \mbox{ in }D_{j},
	\end{cases}
	\]
	for $j=1,2$. Besides, when $\mathcal Bu_j=\mathscr L_A^s u_j$, we further assume $q_j(x) \neq 0$ for $x\in \Omega$ for $j=1,2$. Suppose that $\Lambda_{D_{1},q_{1}}g=\Lambda_{D_{2},q_{2}}g$
	in $\mathcal{O}_2$, for any given nonzero $g\in C_{c}^{\infty}(\mathcal{O}_1)$
	with $u_{1}=u_{2}=g$ in $\Omega_{e}$, then $D_{1}=D_{2}$.\end{thm}
\begin{proof}
	First, we prove that $u_{1}=u_{2}$ in $\mathbb{R}^{n}$ whenever
	$\Lambda_{D_{1},q_{1}}g=\Lambda_{D_{2},q_{2}}g$ in $\mathcal{O}_2$
	and $u_{1}=u_{2}=g$ in $\Omega_{e}$ for the non-identically zero function $g\in C^\infty _c(\mathcal O_1)$. 
	
	Let $w:=u_{1}-u_{2}\in\widetilde{H}^{s}(\Omega)$,
	Then $w$ solves 
	\[
	\begin{cases}
	\mathscr{L}_A^{s}w+q_{1}w=(q_{2}-q_{1})u_{2} & \mbox{ in }\Omega\backslash(\overline{D_{1}}\cup\overline{D_{2}}),\\
	w=0 & \mbox{ in }\Omega_{e}.
	\end{cases}
	\]
	From the condition $\Lambda_{D_{1},q_{1}}g=\Lambda_{D_{2},q_{2}}g$
	in $\mathcal{O}_2$ and $\Lambda_{D_{j},q_{j}}g=\mathscr{L}_A^{s}u_{j}|_{\Omega_{e}}$,
	one can see that 
	\[
	\mathscr{L}_A^{s}w=\mathscr{L}_A^{s}(u_{1}-u_{2})=0\mbox{ in }\mathcal{O}_2\subset\Omega_{e}.
	\]
	In particular, we have $w\in H^{s}(\mathbb{R}^{n})$ such that $w=\mathscr{L}_A^{s}w=0$
	in $\mathcal{O}_2$. By applying the strong uniqueness property (Proposition
	\ref{Prop Strong unique}), we obtain $w\equiv0$ in $\mathbb{R}^{n}$,
	which shows $u_{1}=u_{2}$ in $\mathbb{R}^{n}$. 
	
	Second, we claim that $D_{1}=D_{2}$ in $\mathbb{R}^{n}$ by using
	contradiction arguments. Suppose that $D_{1}\neq D_{2}$. Without
	loss of generality, we can assume that there exists a nonempty open subset $M\Subset D_{2}\backslash\overline{D_{1}}$. Then we have the following two cases.
	
	\textbf{Case 1.}
	\begin{align}\notag
	\begin{cases}
	\text{Either }u_1 =0\text{ in }D_1  \text{ or }\mathscr L_A^su_1 =0 \text{ in }D_1,\\
	u_2=0\text{ in }D_2.
	\end{cases}
	\end{align}
	
	By using the condition $u_{1}=u_{2}$ in $\mathbb{R}^{n}$, we know
	that $u_{1}=u_{2}=0$ in $M\Subset D_{2}$. Applying the nonlocal elliptic
	equation for $u_{1}$ in $M$, it is readily seen that 
	\[
	\mathscr{L}_A^{s}u_{1}=u_{1}=0\mbox{ in }M.
	\]
	Utilizing the strong uniqueness property again, we obtain that $u_{1}\equiv0$
	in $\mathbb{R}^{n}$.
	
	\textbf{Case 2.} 
	\begin{align}\notag
	\begin{cases}
	\text{Either }u_1 =0 \text{ in }D_1 \text{ or }\mathscr L_A^su_1 =0 \text{ in }D_1,\\
	\mathscr L^s_A u_2=0\text{ in }D_2.
	\end{cases}
	\end{align}
	
	Recall that $u_1=u_2$ in $\mathbb R^n$, then $\mathscr L_A^s u_1 =\mathscr L_A^s u_2$ in $\mathbb R^n$ by using a direct calculation. Hence, $\mathscr L_A^s u_1=\mathscr L_A^s u_2 =0$ in $M\Subset D_2\setminus \overline{D_1}$. By using the equation of $u_1$ and $q_1(x)\neq 0$ for $x\in \Omega$, we have 
	\begin{align}\notag
	u_1=\mathscr L_A^s u_1 =0\text{ in }M.
	\end{align}
	Therefore, we have $u_1\equiv 0 $ in $\mathbb R^n$ by the strong uniqueness property.  
	
	However, in either Case 1 or Case 2, the conclusion $u_1\equiv 0$ in $\mathbb R^n$ contradicts to the fact that $u_{1}=g$
	in $\Omega_{e}$ with a non-identically zero exterior data $g$. This proves
	the first part of Theorem \ref{Main Theorem 1} to Theorem \ref{Main Theorem 3} by using a single exterior measurement.

\end{proof}
\begin{rem}
	Indeed, we do not need to use any information about the solution $w$
	in $\Omega\backslash(\overline{D_{1}}\cup\overline{D_{2}})$. We only
	utilize the strong uniqueness of $w$ in the exterior domain $\Omega_{e}$,
	which is a powerful tool in dealing with the nonlocal type inverse
	problems.
\end{rem}

\section{Recovery of the surrounding potential $q$\label{Section 5}}

In this section, we prove the uniqueness in determining the surrounding potential $q$ in $\Omega \setminus \overline{D}$.

\subsection{Runge approximation property}

We shall make essential use of the following Runge approximation property for solutions
of the nonlocal elliptic equation. If $q\in L^{\infty}(\Omega\backslash\overline{D})$
satisfies the eigenvalue condition \eqref{eq:eigenvalue condition},
we denote the solution operator $\Phi_{q}$ by: 
\begin{equation}\notag
\Phi_{q}:\mathbb{H}\rightarrow H^{s}(\mathbb{R}^{n}),g\rightarrow u
\end{equation}
where $\mathbb{H}:=H^{s}(\mathbb{R}^{n})/\widetilde{H}^{s}(\Omega),$
is the abstract space of exterior values, and $u\in H^{s}(\mathbb{R}^{n})$
is the unique solution of $(\mathscr{L}_A^{s}+q)u=0$ in $\Omega\backslash\overline{D}$
with $\mathcal{B}u=0$ in $D$ and $u-g\in\widetilde{H}^{s}(\Omega)$.
\begin{lem}
	\label{lemma:approxiamtion} Let $\Omega\subseteq\mathbb{R}^{n}$
	be a bounded open set and $D\Subset\Omega$ be an open subset. Assume that
	$s\in(0,1)$ and $q\in L^{\infty}(\Omega\setminus\overline{D})$ satisfies
	\eqref{eq:eigenvalue condition}. Let $\mathcal{O}$ be any open
	set of $\Omega_{e}$. Consider the set 
	\[
	\mathbb{A}:=\{\left.u\right|_{\Omega\backslash\overline{D}}:u=\Phi_{q}g,g\in C_{c}^{\infty}(\mathcal{O})\}\cap \{\mathcal Bu=0\text{ in }D\}.
	\]
	Then $\mathbb{A}$ is dense in $L^{2}(\Omega\backslash\overline{D})$. \end{lem}
\begin{proof}
	The proof follows a similar argument to that of \cite[Lemma 5.7]{ghosh2017calderon}. For
	the completeness of this paper, we present a detailed proof in what follows. 
	
	By the Hahn-Banach theorem, it is sufficient to show that for any
	$v\in L^{2}(\Omega\backslash\overline{D})$ with $\int_{\Omega\backslash\overline{D}}vw\,dx=0$
	for all $w\in\mathbb{A}$, then it must satisfy $v\equiv0$ in $\Omega\setminus\overline{D}$.
	If $v$ is a such function, which means $v$ satisfies 
	\begin{equation}
	\int_{\Omega\backslash\overline{D}}v\cdot r_{\Omega\backslash\overline{D}}\Phi_{q}g\,dx=0,\mbox{ for any }g\in C_{c}^{\infty}(\mathcal{O}).\label{eq:equal}
	\end{equation}
	We claim that 
	\begin{equation}
	\int_{\Omega\backslash\overline{D}}v\cdot r_{\Omega\backslash\overline{D}}\Phi_{q}g\,dx=-\mathbb{B}_{q}(\phi,g),\mbox{ for any }g\in C_{c}^{\infty}(\mathcal{O}),\label{eq:claim}
	\end{equation}
	where $\phi\in H^{s}(\mathbb{R}^{n})$ is the solution given by Lemma
	\ref{lemma:equivalence} of 
	\[
	(\mathscr{L}_A^{s}+q)\phi=v\in\Omega\backslash\overline{D},\quad\phi\in\widetilde{H}^{s}(\Omega\backslash\overline{D})
	\]
	In other words, $\mathbb{B}_{q}(\phi,w)=\int_{\Omega\backslash\overline{D}}v\cdot r_{\Omega\backslash\overline{D}}w\,dx$
	for any $w\in\widetilde{H}^{s}(\Omega\backslash\overline{D})$. To
	prove \eqref{eq:claim}, let $g\in C_{c}^{\infty}(\mathcal{O})$,
	and we denote $u_{g}:=\Phi_{q}g\in\widetilde{H}^{s}(\mathbb{R}^{n})$
	such that $u_{g}-g\in\widetilde{H}^{s}(\Omega)$. Then we have 
	\[
	\int_{\Omega\backslash\overline{D}}v\cdot r_{\Omega\backslash\overline{D}}\Phi_{q}g\,dx=\int_{\Omega\backslash\overline{D}}v\cdot r_{\Omega\backslash\overline{D}}(u_{g}-g)\,dx=\mathbb{B}_{q}(\phi,u_{g}-g)=-\mathbb{B}_{q}(\phi,g)
	\]
	in which we have used the fact that $u_{g}$ is a solution and $\phi\in\widetilde{H}^{s}(\Omega\backslash\overline{D})$.
	
	Combining \eqref{eq:equal} and \eqref{eq:claim}, we can obtain that
	\[
	\mathbb{B}_{q}(\phi,g)=0,\mbox{ for any }g\in C_{c}^{\infty}(\mathcal{O})
	\]
	Using the fact that $r_{\Omega\backslash\overline{D}}g=0$, since $g\in C_{c}^{\infty}(\mathcal{O})$
	we can derive that 
	\[
	\int_{\mathbb{R}^{n}}\mathscr{L}_A^{s}\phi\cdot g\,dx=0\mbox{ for any }g\in C_{c}^{\infty}(\mathcal{O}),
	\]
	and thus we obtain that $\phi\in H^{s}(\mathbb{R}^{n})$ satisfies
	\[
	\left.\mathscr{L}_A^{s}\phi\right|_{\mathcal{O}}=\left.\phi\right|_{\mathcal{O}}=0.
	\]
	By the strong uniqueness property again, we know that $\phi\equiv0$
	in $\mathbb{R}^{n}$ and also $v\equiv0$ in $\Omega\backslash\overline{D}$.
	This finishes the proof. 
\end{proof}

\begin{rem}
It is easy to see that the soft or hard condition $\mathcal Bu=0$ in $D$ does not affect the conclusion of the previous lemma.
\end{rem}

\subsection{Proof of the uniqueness in determining $q$}

From the equal DtN maps, by the first statements of Theorems~\ref{Main Theorem 1}--\ref{Main Theorem 3}, we know that the embedded obstacle $D$ is uniquely recovered. Next,
we prove the global uniqueness in determining the potential $q\in L^{\infty}(\Omega\backslash\overline{D})$.

\begin{thm}
	\label{thm:uniqueness of q} For $n\geq2$, let $\Omega$ be a bounded
	open set in $\mathbb{R}^{n}$, $D\Subset\Omega$ be an open subset and and
	$\mathcal{O}_1,\mathcal O_2\subset\Omega_{e}$ be an arbitrary nonempty open set.
	Let $q_j\in L^{\infty}(\Omega\backslash\overline{D})$ satisfy
	\eqref{eq:eigenvalue condition} and $u_{j}\in H^{s}(\mathbb{R}^{n})$ be the unique (weak) solution
	of 
	\[
	\begin{cases}
	\mathscr{L}_A^{s}u_{j}+q_{j}u_{j}=0 & \mbox{ in }\Omega\backslash\overline{D_{j}},\\
	\mathcal Bu_{j}=0  & \mbox{ in }D_{j},
	\end{cases}
	\]
    Assume that $\Lambda_{D,q_{j}}$ are the
	DtN maps with respect to $(\mathscr{L}_A^{s}+q_{j})u_j=0$ for $j=1,2$. If 
	\[
	\left.\Lambda_{D,q_{1}}g\right|_{\mathcal{O}_2}=\left.\Lambda_{D,q_{2}}g\right|_{\mathcal{O}_2}
	\]
	for any $g\in C_{c}^{\infty}(\mathcal{O}_1)$
	with $u_{1}=u_{2}=g$ in $\Omega_{e}$, then one can conclude that 
	\[
	q_{1}=q_{2}\mbox{ in }\Omega\backslash\overline{D}
	\]
\end{thm}
\begin{proof}
	Since $\left.\Lambda_{D,q_{1}}g\right|_{\mathcal{O}_2}=\left.\Lambda_{D,q_{2}}g\right|_{\mathcal{O}_2}$
	for any $g\in C_{c}^{\infty}(\mathcal{O}_1)$, where $\mathcal{O}_1,\mathcal O_2$
	are open sets of $\Omega_{e}$, substituting this condition into
	the integral identity in Lemma \ref{prop:Integral identity}, we have
	\begin{equation}
	\int_{\Omega\backslash\overline{D}}(q_{1}-q_{2})u_{1}u_{2}\,dx=0,\label{eq:newintegral}
	\end{equation}
	where $u_{j}\in H^{s}(\mathbb{R}^{n})$ is the solution of $(\mathscr{L}_A^{s}+q_{j})u_{j}=0$
	in $\Omega\backslash\overline{D}$ with the associated exterior values
	$g_{j}\in C_{c}^{\infty}(\mathcal{O}_j)$, for $j=1,2$ respectively.
	
	Given $\varphi\in L^{2}(\Omega\backslash\overline{D})$, by Proposition
	\ref{lemma:approxiamtion}, suppose there exist a sequences $(u_{j}^{(k)})_{k\in\mathbb{N}}$
	for $j=1,2$ of functions in $H^{s}(\mathbb{R}^{n})$, which satisfy
	\[
	\begin{split} & (\mathscr{L}_A^{s}+q_{1})u_{1}^{(k)}=(\mathscr{L}_A^{s}+q_{2})u_{2}^{(k)}=0\mbox{ in }\Omega\backslash\overline{D},\\
	&\mathcal{B}u_1^{(k)}=0 \text{ and } \mathcal{B}u_2^{(k)}=0 \text{ in D },\\
	& u_{j}^{(k)}=g_{j}^{(k)}\text{ in }\Omega_{e},\text{ for some exterior data }g_{j}^{(k)}\in C_{c}^{\infty}(\mathcal{O}),\\
	& r_{\Omega\backslash\overline{D}}u_{1}^{(k)}=\varphi+r_{1}^{(k)},r_{\Omega\backslash\overline{D}}u_{2}^{(k)}=1+r_{2}^{(k)},\text{ for any }k\in\mathbb{N},
	\end{split}
	\]
	where $r_{1}^{(k)},r_{2}^{(k)}\rightarrow0$ in $L^{2}(\Omega\backslash\overline{D})$
	as $k\rightarrow\infty$. Substituting these solutions into the integral
	identity \eqref{eq:newintegral} and taking the limit as $k\rightarrow\infty$,
	we can deduce that 
	\[
	\int_{\Omega\backslash\overline{D}}(q_{1}-q_{2})\varphi\,dx=0
	\]
	Since $\varphi\in L^{2}(\Omega\backslash\overline{D})$ is arbitrary, we readily see that $q_{1}=q_{2}$ in $\Omega\setminus\overline{D}$. This also completes the second part of Theorems \ref{Main Theorem 1}--\ref{Main Theorem 3}.
\end{proof}

\section*{Acknowledgement}
Y.-H. Lin is partially supported by MOST of Taiwan 160-2917-I-564-048 and he would like to thank Professor Ru-Yu Lai for helpful discussions.
H. Liu was supported by the FRG and startup grants from Hong Kong
Baptist University, and Hong Kong RGC General Research Funds, 12302415
and 12302017.

%\bibliographystyle{plain}
%\bibliography{ref}

\begin{thebibliography}{99}

\bibitem{AR} {G. Alessandrini and L. Rondi}, {\it Determining a sound-soft polyhedral scatterer by a single far-field measurement}, Proc. Amer. Math. Soc., {\bf 35} (2005), 1685--1691. Corrigendum: Preprtint arXiv math.Ap/0601406, 2006.

\bibitem{caffarelli2016fractional}
L.~A. Caffarelli and P.~R. Stinga,
{\it Fractional elliptic equations, {C}accioppoli estimates and
  regularity},
{Annales de l'Institut Henri Poincare (C) Non Linear Analysis},
  {\bf 33} (2016), 767--807.
  
\bibitem{CY}  J. Cheng and M. Yamamoto, {\it Uniqueness in an inverse scattering problem within non-trapping polygonal obstacles with at most two incoming waves}, Inverse Problems, {\bf 19} (2003), 1361--1384.
  
 \bibitem{CK} {D. Colton and R. Kress}, {\it Inverse Acoustic and Electromagnetic Scattering Theory}, 2nd Edition, Springer-Verlag, Berlin, 1998.
  
\bibitem{CSU} C. Kenig, J. Sj\"ostrand and G. Uhlmann, {\it The Calder\'on problem with partial data}, Ann. of Math. (2), {\bf 165} (2007), 567--591.

\bibitem{di2012hitchhiks}
E. D.~Nezza, G. Palatucci and E. Valdinoci,
{\it Hitchhiker's guide to the fractional {S}obolev spaces},
\newblock {\em Bulletin des Sciences Math{\'e}matiques}, {\bf 136} (2012), 521--573.

\bibitem{ghosh2017calderon}
T. Ghosh, Y.-H. Lin and J. Xiao,
{\it The {C}alder{\'o}n problem for variable coefficients nonlocal
  elliptic operators}, {Communications in Partial Differential Equations},
  {\bf 42(12)} (2017), 1923-1961.

\bibitem{ghosh2016calder}
T. Ghosh, M. Salo and G. Uhlmann,
{\it The {C}alder{\'o}n problem for the fractional {S}chr{\"o}dinger
  equation},
 {arXiv:1609.09248}. 

\bibitem{harrach2017monotonicity}
B. Harrach and Y.-H. Lin, {\it Monotonicity-based inversion of the fractional {S}chr\"{o}dinger
  equation}, { arXiv:1711.05641}.
  
\bibitem{HNS} N. Honda, G. Nakamura and M. Sini, {\it Analytic extension and reconstruction of obstacles from few measurements for elliptic second order operators}, Math. Ann., {\bf 355} (2013), 401--427. 

\bibitem{IUY} O. Imanuvilov, G. Uhlmann and M. Yamamoto, {\it The Calder\'on problem with partial data in two dimensions}, J. Am. Math. Soc., {\bf 23} (2010), 655--691. 

\bibitem{Isa2} {V. Isakov}, {\it Inverse Problems for Partial Differential Equations}, 2nd edition, Applied Mathematical Sciences, 127, Springer-Verlag, New York, 2006.

\bibitem{KL} A. Kirsch X. Liu, {\it Direct and inverse acoustic scattering by a mixed-type scatterer}, Inverse Problems, {\bf 29} (2013), 065005.

\bibitem{KP} A. Kirsch and L. P\"aiv\"arinta, {\it On recovering obstacles inside inhomogeneities}, Math. Meth. Appl. Sci., {\bf 21} (1998), 619--651. 

\bibitem{lai2017global}
R.-Y. Lai and Y.-H. Lin,
{\it Global uniqueness for the semilinear fractional {S}chr\"{o}dinger
  equation}, {arXiv:1710.07404}.
  
  \bibitem{LiuLiu} H. Liu and X. Liu, {\it Recovery of an embedded obstacle and its surrounding medium from formally determined scattering data}, Inverse Problems, {\bf 33} (2017), 065001.
  
\bibitem{LMRX}  H. Liu, M. Petrini, L. Rondi and J. Xiao, {\it Stable determination of sound-hard polyhedral scatterers by a minimal number of scattering measurements}, J. Differential Equations, {\bf 262} (2017), 1631--1670.

\bibitem{LiuZhaoZou} H. Liu, H. Zhao and C. Zou, {\it Determining scattering support of anisotropic acoustic mediums and obstacles}, Commun. Math. Sci., {\bf 13} (2015), 987--1000.

\bibitem{LiuZhang} X. Liu and B. Zhang, {\it Direct and inverse obstacle scattering problems in a piecewise homogeneous medium}, SIAM J. Appl. Math., {\bf 70} (2010), 3105--3120.  
  
\bibitem{Liu1} {H. Liu and J. Zou}, {\it Uniqueness in an inverse acoustic scatterer scattering problem for both sound-hard and sound-soft polyhedral scatterers}, Inverse Problems, {\bf 22} (2006), 515--524.

\bibitem{Liu2} {H. Liu and J. Zou}, {\it On unique determination of partially coated polyhedral scatterers with far-field measurements}, Inverse Problems, {\bf 23} (2007), 297--308.

\bibitem{LZsurvey} {H. Liu and J. Zou}, {\it On uniqueness in inverse acoustic and electromagnetic obstacle scattering problems}, J. Phys.: Conf. Ser., {\bf 124} 012006. 


\bibitem{mclean2000strongly}
W. McLean,
\newblock {Strongly Elliptic Systems and Boundary Integral Equations},
\newblock Cambridge University Press, 2000.

\bibitem{O} S. O'Dell, {\it Inverse scattering for the Laplace-Beltrami operator with complex electromagnetic potentials and embedded obstacles}, Inverse Problems, {\bf 22} (2006), 1579--1603. 

\bibitem{pazy2012semigroups}
A. Pazy,
\newblock {Semigroups of Linear Operators and Applications to Partial
  Differential Equations}, Volume~44,
\newblock Springer Science \& Business Media, 2012.

\bibitem{Ron1} {L. Rondi}, {\it Unique determination of non-smooth sound-soft scatterers by finitely many far-field measurements}, Indiana Univ. Math. J., {\bf 52} (2003), 1631--1662.

\bibitem{Ron2} {L. Rondi}, {\it Stable determination of sound-soft polyhedral scatterers by a single measurement}, Indiana Univ. Math. J., {\bf 57 } (2008), 1377--1408.

\bibitem{rudin1991functional}
W. Rudin,
\newblock { Functional Analysis},
\newblock McGraw-Hill Series in Higher Mathematics. New York: McGraw-Hill Book
  Company.

\bibitem{ruland2017exponential}
A. R{\"u}land and M. Salo,
{\it Exponential instability in the fractional {C}alder\'{o}n problem},
{arXiv:1711.04799}.

\bibitem{ruland2017fractional}
A. R{\"u}land and M. Salo,
{\it The fractional {C}alder\'{o}n problem: low regularity and stability},
 { arXiv:1708.06294}.

\bibitem{stinga2010extension}
P.~R. Stinga and J.~L. Torrea, {\it Extension problem and {H}arnack's inequality for some fractional
  operators},
\newblock {Communications in Partial Differential Equations},
  {\bf 35} (2010), 2092--2122.

\end{thebibliography}

\end{document}